\documentclass{amsart}

\title{Prime Torsion in the Brauer Group of an Elliptic Curve} 

\usepackage[foot]{amsaddr}
\author{Charlotte Ure} 
\address{Department of Mathematics, University of Virginia, Charlottesville, VA 22904, U.S.A.\\ ORCiD 0000-0003-4484-956X}
\email{cu9da@virginia.edu}

\usepackage{mathscinet}
\usepackage{amssymb,amsmath, amsthm,cleveref}
\newtheorem{thm}{Theorem}[section] 
\crefname{thm}{theorem}{theorems}

\crefname{algo}{algorithm}{algorithms}

\crefname{claim}{claim}{claims}

\crefname{conj}{conjecture}{conjectures}
\newtheorem{cor}[thm]{Corollary}
\crefname{cor}{corollary}{corollaries}

\crefname{defn}{definition}{definitions}

\crefname{exer}{exercise}{exercises}

\crefname{Ex}{example}{examples}

\crefname{lemma}{lemma}{lemmas}
\newtheorem{lem}[thm]{Lemma}
\crefname{lem}{lemma}{lemmas}
\newtheorem{prop}[thm]{Proposition}
\crefname{prop}{proposition}{propositions}

\crefname{prob}{problem}{problems}
\newtheorem{rem}[thm]{Remark}
\crefname{rem}{remark}{remarks}

\crefname{ques}{question}{questions}

\crefname{goal}{goal}{goals}

\usepackage{cleveref}

\usepackage[numbers]{natbib}
\usepackage[color,arrow,matrix,curve]{xy}
\usepackage{mathtools}
\usepackage{xcolor}

\DeclareMathOperator{\Div}{Div}
\DeclareMathOperator{\divisor}{div}
\DeclareMathOperator{\inflation}{inf}
\DeclareMathOperator{\res}{res}
\DeclareMathOperator{\Cor}{cor}
\DeclareMathOperator{\Spec}{Spec}

\DeclareMathOperator{\Pic}{Pic}

\DeclareMathOperator{\Gal}{Gal}

\DeclareMathOperator{\Br}{Br}

\def\etale{\'etale\ }
\def\et{\'et}
\def\tor[#1][#2]{\prescript{}{#1}{#2}}
\def\kbar{\overline{k}}
\def\ra{\rightarrow}

\def\ZZ{\mathbb{Z}}
\def\Z{\mathbb{Z}}
\def\Q{\mathbb{Q}}
\DeclareMathOperator{\Prin}{Prin}

\DeclareMathOperator{\e}{e}

\DeclareMathOperator{\chark}{char}
\renewcommand{\H}{H}

\begin{document} 
\begin{abstract}
        We give an algorithm to explicitly determine all elements of the $q$-torsion (for $q$ an odd prime) of the Brauer group of an elliptic curve over any base field of characteristic different from $q$, containing a primitive $q$-th root of unity. These elements of the Brauer group are given as tensor products of symbol algebras over the function field of the elliptic curve. We give sufficient conditions to determine if the Brauer classes that arise are trivial. Using our algorithm, we derive an upper bound on the symbol length of the prime torsion of $\Br(E)/\Br(k)$. 
\end{abstract}
	
\maketitle

\section{Introduction} 
	
The Brauer group of a variety is an important invariant and has applications in both arithmetic and algebraic geometry. For example, this invariant has found applications in algebraic geometry via obstruction to rationality and in arithmetic geometry via obstruction to the existence of rational points. In a famous construction, Artin and Mumford used the Brauer group of the function field of $\mathbb{P}^2$ to give a counterexample to the L\"uroth problem; they constructed a complex unirational $3$-fold which is not rational \cite{artinmumford}. Furthermore, Manin described an obstruction to the Hasse principle for varieties that lies in the Brauer group (the Brauer-Manin obstruction) \cite{Manin}. There has been an ongoing effort to understand this obstruction in more detail for various varieties, see for example \cite{Berg-Varilly-Alvarado}, \cite{Creutz-Viray-Voloch}, and \cite{Skorobogatov:Beyond-the-Manin-Obstruction}. \\

The goal of the present paper is to gain an explicit description of the Brauer group of an elliptic curve $E$. It will be sufficient to examine $d$-torsion subgroups of a Brauer group as it is a torsion abelian group. Fix an integer $d\geq 2$, coprime to the characteristic of $k$, and suppose that $k$ contains a primitive $d$-th root of unity $\rho$. We want to explicitly describe the $d$-torsion of the Brauer group $\tor[d][\Br(E)]$. By absolute purity \cite{colliot-thelene-sansuc:therationalityproblem,Gabber:ANoteOnTheUnramifiedBrauerGroupAndPurity}, the Brauer group of $E$ is naturally isomorphic to the unramified Brauer group of $k(E)$, the function field of $E$. In particular, this means that elements in $\tor[d][\Br(E)]$ can be represented by central simple algebras over $k(E)$. A favorite set of generators of the Brauer group are symbol (or cyclic) algebras. These are $k(E)$-algebras with two generators $x$ and $y$ subject to the relations
$x^d= a, y^d = b,$ and $xy=\rho yx$ for some $a,b \in k(E)^\times$.  
The Merkurjev-Suslin theorem \cite{merkurjev-suslin} implies that every element in $\tor[d][\Br(E)]$ can be described by a tensor product of these symbol algebras over $k(E)$. We aim to describe the elements of $\tor[d][\Br(E)]$ in terms of these products.  \\

The subgroup $\tor[2][\Br(E)]$ was described in \cite{chernousov-guletskii:2-torsion}, \cite{Guletskii1997}, \cite{GuletskiiV.Yanchevskii1998}, and \cite{Pumplun1998}, where the authors compute its elements, and determine when such a description is trivial, using quaternion algebras over any base field of characteristic not equal to two. In a vast generalization, using a combination of techniques from the 2-torsion case as well as the work of Skorobogatov giving an abstract description, we calculate the elements of $\tor[d][\Br(E)]$, and conditions for when such an element is trivial, in the following cases:
\begin{enumerate} 
	\item[(A)] any $d$: When $\chark k$ is coprime to $d$, $k$ contains a primitive $d$-th root of unity, and the $d$-torsion of $E(\kbar)$ is $k$-rational (see \Cref{Prop:main result in split case}). 
	\item[(B)] $d=q$, an odd prime: When $\chark k$ is coprime to $q$ and $k$ contains a primitive $q$-th root of unity (see \Cref{sec:Algorithm}). 
\end{enumerate} 
The description of elements of $\tor[d][\Br(E)]$ in (A) was previously known (see for instance \cite[Chapter 4]{skorobogatov} and \cite[Remark 6.3]{Chernousov2016}). We include the computation in this paper for completeness. Generalizing the description of $\tor[2][\Br(E)]$, the two-torsion of the Brauer group of a hyperelliptic curve was given in \cite{Creutz-Viray:TwoTorsionintheBrauerGroupofaHyperellipticCurve}. Finally, elements in the relative Brauer group of a curve $X$ of genus $1$ over $k$ were described in \cite{Ciperiani-Krashen:RelativeBrauerGroupsofGenus1Curves} and \cite{Han:RelativeBrauerGroupsOfFunctionFieldsOfCurvesofGenusone}. This relative Brauer group is the kernel of the map $\Br k \rightarrow \Br k(X)$. In the case that $X=E$ is an elliptic curve, the relative Brauer group is trivial since $\Br(k) \hookrightarrow\Br(E) \hookrightarrow\Br (k(E)).$  \\

Consider the Hochschild-Serre spectral sequence $$\H^i\left( k, \H^j \left(\overline{E}, \mathbb{G}_m \right) \right) \Rightarrow \H^{i+j}\left(E, \mathbb{G}_m \right).$$
The $d$-torsion of its sequence of low degree terms gives the short exact sequence
\begin{align}\label{Eq:exact sequence on tor[d]Br(E)} 
\xymatrix{0 \ar[r]& \tor[d][\Br(k)] \ar[r]^i & \tor[d][\Br(E)] \ar[r]^(.4)r & \tor[d][\H^1\left(k, E\left(\kbar\right)\right)]  \ar[r] & 0 },
\end{align}
where the first map sends the class of a central simple algebra $A$ to the class of $A \otimes k(E)$ under the identification $\Br E \subset \Br k(E)$. For more details on this sequence, see also \cite{Faddeev:SimpleAlgebrasOverAFieldOfAglebraicFunctionsOfOneVariable} and \cite{Lichtenbaum1969}. We will discuss the second map in more detail in \Cref{Mkrational}.
Denote the zero-section of $E$ by $s: \Spec(k) \ra E$. Sequence (\ref{Eq:exact sequence on tor[d]Br(E)}) is split on the left by the induced specialization $s^*: \Br(E) \ra \Br(k)$.
Thus we get the decomposition $$\tor[d][\Br(E)] = \tor[d][\Br(k)] \oplus \tor[d][\H^1\left(k, E\left(\kbar\right)\right)].$$ In order to calculate $\tor[d][\Br(E)]$ it is sufficient to describe a splitting to $r$ in sequence (\ref{Eq:exact sequence on tor[d]Br(E)}). To that end, consider the Kummer sequence 
\begin{align*}
\xymatrix{0 \ar[r]& M \ar[r]& E(\kbar) \ar[r]^{[d]}& E(\kbar) \ar[r] & 0 },
\end{align*} 
where $[d]$ denotes multiplication by $d$ on the elliptic curve and $M$ is the full $d$-torsion of $E$. This induces a short exact sequence on group cohomohology
\begin{align}\label{eq:Kummer induces on cohomology} 
\xymatrix{ 0 \ar[r] & E(k)/[d]E(k) \ar[r]^\delta & \H^1(k,M) \ar[r]^(0.45){\lambda} & \tor[d][\H^1\left(k,E(\kbar)\right)] \ar[r]& 0 }.
\end{align} 
We want to relate the sequences (\ref{Eq:exact sequence on tor[d]Br(E)}) and (\ref{eq:Kummer induces on cohomology}). 
Define $\epsilon: \H^1(k,M) \ra \tor[d][\Br(E)]$ as the following composition 
\begin{equation}\label{Eq:epsilon}
\mbox{\large{$\epsilon$ :}}\left\{ 
\begin{gathered}
\xymatrix{
	\H^1(k,M) \ar[r]^(.35)\sim & \H^1_{\text{\et}}\left(\Spec(k),M\right) \ar[r]^(.6){p^*}& \H^1_{\text{\et}}(E,M)}\\
\xymatrix{\ar[r]^(.2){- \cup [\mathcal{T}]}& \H^2_{\text{\et}}(E,M \otimes M) \ar[r]^(.55)\e& \H^2_{\text{\et}}\left(E, \mu_d \right)  \ar[r]& \tor[d][\Br(E)] },
\end{gathered}\right.
\end{equation} 
where 
\begin{itemize} 
	\item $p: E \ra \Spec k$ is the structure map,
	\item $\mathcal{T}$ is the torsor given by multiplication by $d$ on $E$, and 
	\item $e$ is the map induced by the Weil-pairing. 
\end{itemize} In \cite[Chapter 4]{skorobogatov}, the author proves, using abstract methods, that $\epsilon$ induces a splitting of squence (\ref{Eq:exact sequence on tor[d]Br(E)}). We reprove this result using explicit methods and compute $\epsilon$ in terms of group cohomology. The main obstacle to this explicit description is the \etale cup-product with the torsor $\mathcal{T}$. There is an ongoing effort to calculate these products explicitly (see for example \cite{Aldrovandi-Ramachandra:Cup-products-the-Heisenberg-group-and-codimension-two-algebraic-cycles} and \cite{Poonen-Rains:SelCupProducts}). In this paper, we overcome this difficulty by calculating the cup product at the generic point and using results from Galois cohomology. We use our explicit description of $\epsilon$ to give a complete description of $\tor[d][\Br(E)]$ under the assumption that $M$ is $k$-rational. 
\begin{thm}\label{Prop:main result in split case} 
	Suppose that the $d$-torsion $M$ of $E(\kbar)$ is $k$-rational. Fix two generators $P$ and $Q$ of $M$. Denote by $0$ the identity of the group law on $E$. Let $t_P,t_Q \in k(E)$ with divisors  $\divisor(t_P) = d(P) - d(0)$ and $\divisor(t_Q) = d(Q)- d(0)$ so that $t_P\circ [d], t_Q \circ [d] \in (k(E)^\times)^d$. Then the $d$-torsion of $\Br(E)$ decomposes as 
	$$\tor[d][\Br(E)] =\tor[d][\Br(k)] \oplus I$$ and every element in $I$ can be represented as a tensor product 
	$$\left( a, t_P \right)_{d,k(E)} \otimes \left( b, t_Q \right)_{d,k(E)}$$ with $a,b \in k^\times$. Such a tensor product is trivial if and only if one of the following equations hold in $k^\times/(k^\times)^d$:  
	\begin{itemize} 
	    \item $a \equiv 1$ and $b\equiv 1,$
	    \item $a \equiv t_Q(P)$ and $b \equiv \frac{t_P(P+Q)}{t_P(Q)}$, 
		\item $a \equiv \frac{ t_Q(P+Q)}{t_Q(P)}$ and $b=t_P(Q)$, or 
		\item $a \equiv t_Q (R)$ and $b\equiv t_P(R)$ for some $R\in E(k), R \neq 0,P,Q$.\end{itemize}  
\end{thm}
In addition to recovering the elements of $\tor[d][\Br(E)]$ as calculated in \cite[Remark 6.3]{Chernousov2016}, we give sufficient and necessary conditions for those elements to be trivial. For the proof of \Cref{Prop:main result in split case}, see \Cref{Mkrational}.
\\

The main result of this paper is the algorithm given in \Cref{sec:Algorithm}. Let $q$ be an odd prime coprime to the characteristic of $k$. The algorithm can be used to determine the $q$-torsion of $\Br(E)$ over an arbitrary field $k$ containing a primitive $q$-th root of unity. This algorithm is proved in \Cref{Mnotkrational}. The main idea of the proof is the following: Let $L$ be the smallest Galois extension of $k$ so that the $q$-torsion $M$ of $E(\kbar)$ is $L$-rational. Use \Cref{Prop:main result in split case} to describe the Brauer group of $E \times_{\Spec k} \Spec L$. Note that the order of $L$ over $k$ divides the order of $SL_2\left( \mathbb{F}_q \right)$, which is $q (q-1) (q+1)$. If $q$ does not divide the order of $L$ over $k$, we use the fact that restriction followed by corestriction is an isomorphism to determine $\Br(E)$. If the order of $L$ over $k$ equals $q$, we apply the inflation restriction exact sequence to our problem. Finally, if $q$ divides the order of $\Br(E)$, we combine the two previous cases. \\

A direct consequence of the algorithm is the following result on the symbol length. 

\begin{cor}Let $q$ be an odd prime coprime to the characteristic of $k$ and assume that $k$ contains a primitive $q$-th root of unity. For any elliptic curve $E$ over $k$, the $q$-torsion of the Brauer group decomposes as $\tor[q][\Br(E)] = \tor[q][\Br(k)] \oplus I$ and every element in $I$ can be written as a tensor product of at most $n_q = 2(q-1)(q+1)$ symbol algebras over $k(E)$. This means that the symbol length in $\tor[q][\Br(E)]/\tor[q][\Br(k)]$ is at most $n_q$.
\end{cor} 

Finally, we give multiple examples in \Cref{ch:Examples}. For additional examples, see also \cite{dissertation}. \\

\noindent\textbf{Notation.} 
Throughout this paper, $k$ denotes a base field and $E$ is an elliptic curve over $k$. We denote the addition on $E$ by $\oplus$, the subtraction by $\ominus$, and the point at infinity by $0$. Let $d\geq 2$ be an integer coprime to the characteristic of $k$ so that $k$ contains a primitive $d$-th root of unity. $M$ denotes the $d$-torsion of $E(\kbar)$, and $[d]$ the multiplication by $d$-map on $E$. $P$ and $Q$ will always be generators of $M$ and $\e(.,.)$ the Weil-pairing on $E$ with $\e(P,Q) = \rho$. Fix an isomorphism $[.]_\rho: \mu_{d} \rightarrow \Z/d\Z$ with $\left[\rho^i\right]_\rho = i $. Furthermore, identify $\Z/d\Z$ with the subset $\{0, \ldots, d-1\}$ of the integers and denote the image of $\rho^i$ under the composition by $\left[\rho^i\right]_\rho^\Z= i \in \Z$. When we assume that $d$ is an odd prime, we write $d=q$. 

For a field $F$, denote an algebraic closure by $\overline{F}$ and its absolute Galois group by $G_F$. Furthermore, $H^i(F,A) = H^i(G_F,A)$ is the $i$-th (profinite) group cohomology group of $G_F$ with coefficients in a $G_F$-module $A$. $\res, \Cor,$ and $\inflation$ denote the restriction, the corestriction, and the inflation map, respectively. \\

\noindent\textbf{Acknowledgements.}
I would like to thank my advisor Rajesh Kulkarni, for his expertise and support throughout the process of writing this paper. I also thank Igor Rapinchuk for helpful conversations and Alexei N. Skorobogatov for a useful correspondence. I thank David Saltman and Adam Chapman for comments on an earlier draft of this article. Finally, I would like to thank Mateo Attanasio, Caroline Choi, and Andrei Mandelshtam for comments on a previous version. I also thank the anonymous referee for their thoughtful comments and suggestions. 

\section{The Algorithm}\label{sec:Algorithm}
Let $q$ be an odd prime and let $k$ be a field of characteristic different from $q$ that contains a primitive $q$-th root of unity $\rho$. Let $E$ be an elliptic curve over $k$. Denote by $M$ the $q$ torsion of $E(\kbar)$ with generators $P$ and $Q$ so that $\e(P,Q)=\rho$. The Brauer group of $E$ decomposes as 
$$\tor[q][\Br(E)] = \tor[q][\Br(k)] \oplus I.$$
Below, we describe an algorithm to find the elements in $I$ and give sufficient conditions for when they are trivial: 
First, determine the kernel of the natural Galois representation 
$$\Psi: G_k \rightarrow GL_2\left(\mathbb{F}_q\right).$$
Since $k$ contains a primitive $q$-th root of unity, the image of $\Psi$ lies in $SL_2(\mathbb{F}_q)$ (for more details see \Cref{Mnotkrational}). Denote by $L$ the fixed field of this kernel and fix some elements $t_P,t_Q \in L(E)$ with $\divisor(t_P) = q(P) - q(0)$ and $\divisor(t_Q) = q(Q) - q(0)$ so that $t_P\circ [q], t_Q \circ [q] \in (k(E)^\times)^q$.. \\

Suppose first that $q$ does not divide the order of $L/k$. Then every element in $I$ can expressed as 
$$\Cor_{L(E)/k(E)} \left( a, t_P\right)_{L(E)} \otimes \Cor_{L(E)/k(E)}\left( b, t_Q\right)_{L(E)}$$
for some $a,b \in L^\times$. The corestriction may be computed explicitly using an algorithm by Rosset and Tate \cite{Rosset1983}. Furthermore, such a tensor product is trivial if one of the following conditions hold in $L^\times/(L^\times)^q \times L^\times/(L^\times)^q$:
\begin{equation}\label{eq:rel}
\begin{cases} 
            a \equiv 1 \text{ and } b\equiv 1,\\
            a \equiv t_Q(P) \text{ and }b \equiv \frac{t_P(P+Q)}{t_P(Q)},\\
		    a \equiv \frac{ t_Q(P+Q)}{t_Q(P)} \text{ and } b=t_P(Q), \text{ or} \\
		    a \equiv t_Q (R) \text{ and } b\equiv t_P(R) \text{ for some }R\in E(k), R \neq 0,P,Q.
\end{cases}
\end{equation} 

On the other hand, if $q$ divides the order of $L/k$, fix some intermediate field $L'$ so that $L/L'$ is a Galois extension of degree $q$. We may choose $P$ and $Q$ so that $\Gal(L/L')$ is generated by $\overline{\sigma}$ with $\sigma(P) = P$ and $\sigma(Q) = P\oplus Q$. Fix $n_Q \in L'(E)$ with $\divisor(n_Q) = \sum_{i=0}^{q-1} \sigma^i(Q) = \sum_{i=0}^{q-1} (iP+Q).$ Then every element in $I$ can be expressed as 
$$\Cor_{L'(E)/k(E)}\left( \left(l^q\right)^i, n_Q \right)_{L'(E)} \otimes \Cor_{L'(E)/k(E)}(a,t_P)_{L'(E)}$$ for some $a \in L'^\times$ and $0 \leq i <q.$ 
Furthermore, an element as above is trivial if 
\begin{enumerate}
    \item the pair $(a,1)$ satisfies one of the conditions in (\ref{eq:rel}) in $L^\times/(L^\times)^q \times L^\times/(L^\times)^q$ and
    \item $i=0$ or the quotient $\frac{E(k) \cap [q]E(L)}{[q]E(k)}$ is not trivial. 
\end{enumerate}
Remark that an element in $I$ may be trivial even if it does not occur in the above list as additional relations may arise from the fact that the corestriction map is not injective. These relations require a more careful treatment, for an example see \Cref{subsec:Ex degree L/k coprime to 3}. 

\section{Background}\label{ch:background}
Let $E$ be an elliptic curve defined over a field $k$. If $k$ is of characteristic different from $2$ and $3$, then $E$ may be described by the affine equation
$$y^2 = x^3 + Ax + B$$
with $A,B \in k$. Denote the point addition on $E(\kbar)$ by $\oplus$, the point subtraction by $\ominus$, and the point at infinity by $0$. Furthermore, for any natural number $d$, we denote by $[d]$ the multiplication by $d$ map on $E(\kbar)$. \\

We now proceed to describe the correspondence between torsors and elements in the first cohomology group. For more details we refer the reader to \cite{skorobogatov}. Let $A$ be an algebraic group defined over a field $F$. An $F$-torsor under $A$ is a non-empty $F$-variety $T$ equipped with a right-action of $A$ so that $T(\overline{F})$ is a principal homogeneous space under $A(\overline{F})$. There is a bijection  
$$ H^1(F,A) \leftrightarrow \left\{ \begin{gathered} \text{$F$-torsors under $A$}\\\text{ up to isomorphism }\end{gathered} \right\}$$
that is explicitly given as follows: Let $T$ be an $F$-torsor under $A$. Choose an $\overline{F}$-point $x_0$ of $T$. By the definition of $F$-torsor, for any $\sigma \in G_F$, there exists a unique $a_\sigma \in A(\overline{F})$ so that 
$\sigma(x_0) = x_0. a_\sigma$. The map $\sigma \mapsto a_\sigma$ defines the cocycle in $H^1(F,A)$ corresponding to $T$. 
Now consider the more general case, where $X$ is an abelian variety over a field $k$. An $X$-torsor under an $X$-group scheme $\mathcal{A}$ is a scheme $\mathcal{T}$ over $X$ together with an $\mathcal{A}$-action compatible with the projection to $X$ that is \'etale-locally trivial. As before, there is a one-to-one correspondence between $X$-torsors under $\mathcal{A}$ and elements of the \etale cohomology $H^1_{\text{\et}}(X,\mathcal{A})$.  
A $d$-covering of an abelian variety $X$ is a pair $(\mathcal{T},\psi)$, where $\mathcal{T}$ is a $k$-torsor under $X$ and $\psi: \mathcal{T} \ra X$ is a morphism such that $\psi(x.t) = d x + \psi(t)$ for any $t \in \mathcal{T}(\kbar), x \in X(\kbar).$ By \cite[Proposition 3.3.4 (a)]{skorobogatov}, any $d$-covering is an $X$-torsor under the $d$-torsion $\tor[d][X]$.\\

In the following, we describe the  correspondence between the Brauer group and the second cohomology group. Let $F$ be a field and let $d\geq 2$ be an integer. We say that two central simple algebras $A$ and $B$ are Morita equivalent if there exist some natural numbers $n$ and $m$ so that $A \otimes M_n(F)$ and $B \otimes M_m(F)$ are isomorphic as $F$-algebras. Elements in the Brauer group are given by equivalence classes of central simple algebras modulo Morita equivalence and the group structure is given by the tensor product. There is a group isomorphism between $\Br(F)$ and $H^2(F, \overline{F}^\times)$ that can be described as follows: Let $K/F$ be a finite Galois extension with Galois group $G$ and let $f$ be a cocycle representing an element in $H^2(G,K^\times)$. Consider the $F$-vector space 
$A = F \left<x_g : g \in G \right>$ with multiplication $\lambda x_g = x_g g(a)$ and $x_g x_h = f(g,h) x_{gh}$. This turns $A$ into a finite dimensional central simple algebra over $F$. 

From now on suppose that the field $F$ contains a primitive $d$-th root of unity $\rho$. Fix an isomorphism $[.]_\rho: \mu_{d} \rightarrow \Z/d\Z$ with $\left[\rho^i\right]_\rho = i $. Furthermore, identify $\Z/d\Z$ with the subset $\{0, \ldots, d-1\}$ of the integers and denote the image of $\rho^i$ under the composition by $\left[\rho^i\right]_\rho^\Z= i \in \Z$. For $a,b \in F^\times$, the symbol algebra 
	\begin{equation}\label{defn:symbolalgebra} 
	(a,b)_{d,F}= (a,b)_{d,F,\rho} := F \left< x,y : x^d = a , y^d = b, xy = \rho yx \right>
\end{equation} 
is a central simple algebra over $F$. The element in $H^2(F,\overline{F}^\times)$ corresponding to $(a,b)_{d,F,\rho}$ can be represented by the cocycle
\begin{equation}\label{eqn:cocyclesymbolalgebra}  
	(\gamma, \tau) \mapsto \begin{cases} a & \text{if }\left[ \frac{\gamma\left( \sqrt[d]{b} \right)}{\sqrt[d]{b}}\right]_\rho^\Z + \left[ \frac{\tau\left( \sqrt[d]{b} \right)}{\sqrt[d]{b}}\right]_\rho^\Z \geq d \\ 1 & \text{else} \end{cases}.
\end{equation} 
For more details we refer the reader to \cite[Chapter 7 \S 29]{Reiner:Maximal-orders}. The following cocycle representing the symbol algebra $(a,b)_{d,F}$ will prove more useful for our purposes.

\begin{prop}\label{prop:sumbolalgebra-cocycle} 
	Let $M$ be the $d$-torsion of an elliptic curve $E$ over $k$ with generators $P$ and $Q$. Assume that the Weil-pairing satisfies $\e(P,Q)=\rho$. Let $F$ be a field containing a primitive $d$-th root of unity $\rho$. The symbol algebra $(a,b)_{d,F}$ can be represented by the cocycle 
	$$(\gamma,\tau) \mapsto \e \left(\left[\frac{\gamma\left( \sqrt[d]{a} \right)}{\left( \sqrt[d]{a} \right)}\right]_\rho P, \left[\frac{\tau\left( \sqrt[d]{b} \right)}{\left( \sqrt[d]{b} \right)}\right]_\rho Q \right)^{-1}$$
	for every pair $a,b \in F^\times$. 
	We will often consider the case $F = k(E)$. 
\end{prop}

\begin{proof}
	Consider the map $$g: \gamma \rightarrow {\sqrt[d]{a}}^{ \left[ \frac{\gamma\left( \sqrt[d]{b} \right)}{\left( \sqrt[d]{b}\right)}\right]_\rho^\Z }.$$
	Remark that since $\e(P,Q) = \rho$ and the Weil pairing is bilinear, we have
	\begin{align*} 
		\frac{\gamma \left( g(\tau)\right)}{g(\tau)}  
		&= \left( \frac{ \gamma\left( {\sqrt[d]{a}} \right)}{\sqrt[d]{a}}\right)^{ \left[ \frac{\tau\left( \sqrt[d]{b} \right)}{\left( \sqrt[d]{b}\right)}\right]_\rho^\Z } = \e \left(\left[\frac{\gamma\left( \sqrt[d]{a} \right)}{\left( \sqrt[d]{a} \right)}\right]_\rho P, \left[\frac{\tau\left( \sqrt[d]{b} \right)}{\left( \sqrt[d]{b} \right)}\right]_\rho Q \right)  
	\end{align*} 
	Furthermore, a direct computation gives
	\begin{align*} 
		\frac{ g(\tau) g(\gamma)}{g(\gamma\tau)} = \begin{cases} a & \text{if }\left[ \frac{\gamma\left( \sqrt[d]{b} \right)}{\sqrt[d]{b}}\right]_\rho^\Z + \left[ \frac{\tau\left( \sqrt[d]{b} \right)}{\sqrt[d]{b}}\right]_\rho^\Z \geq d \\ 1 & \text{else} \end{cases}.
	\end{align*}
	Therefore, subtracting the differential 
	$$dg(\gamma,\tau) = \frac{\gamma \left( g(\tau)\right) g(\gamma)}{g(\gamma\tau)} = \frac{ \gamma\left( g(\tau)\right)}{g(\tau)} \frac{ g(\tau)g(\gamma)}{g(\gamma\tau)}$$
	from the cocycle in (\ref{eqn:cocyclesymbolalgebra}) gives the desired result.
\end{proof}

\section{$M$ is $k$-rational}\label{Mkrational}

Let $d \geq 2$ be an integer coprime to the characteristic of the base field $k$. Assume additionally that $k$ contains a primitive $d$-th root of unity $\rho$. Fix an isomorphism $[.]_\rho: \mu_{d} \rightarrow \Z/d\Z$ with $\left[\rho^i\right] = i $. Furthermore, for $\rho^i \in \mu_d$, let $\left[\rho^i \right]_\rho^\Z = i \in \left\{ 0, \ldots, d-1 \right\}\subset \Z$. Let $E$ be an elliptic curve over $k$ and denote its $d$-torsion by $M$. Assume throughout this section that $M$ is $k$-rational. Fix two generators $P$ and $Q$ of $M$. Denote by $\e(.,.)$ the Weil pairing and assume that $\e(P,Q) = \rho$. Let $t_P, t_Q \in \kbar(E)$ with $\divisor(t_P) = d(P) - d(0)$ and $\divisor(t_Q) = d(Q) - d(0)$. We may assume that $t_P, t_Q \in k(E)$ since their divisors are invariant under the Galois action of $G_k$. Let $\mathcal{T}$ be the torsor given by multiplication by $d$ on $E$.

\begin{prop}\label{Prop:torsor mult by 3 over kbar}
	The pull-back  $\eta^*\left( \mathcal{T}\right)$ along the generic point $\eta: \Spec k(E) \ra E$ corresponds to the element in $H^1\left(k(E),M\right)$ given by the cocycle
	$$\gamma \mapsto \left[ \frac{\gamma\left(\alpha_Q\right)}{\alpha_Q} \right]_\rho P \ominus \left[\frac{\gamma\left(\alpha_P\right)}{\alpha_P} \right]_\rho Q,$$
	where $\alpha_P, \alpha_Q \in \overline{k(E)}$ with $\alpha_P^d = t_P$, $\alpha_Q^d = t_Q$. 
\end{prop}

\begin{proof}
	For the correspondence between torsors and elements in the first cohomology group see \Cref{ch:background}. Let $P' \in E(\kbar)$ so that $[d]P' = P$. Then there is some $g_P \in \kbar(E)$ with $$\divisor\left(g_P \right)= [d]^* (P) - [d]^* (0) = \sum_{R \in M} \left((P' \oplus R) - (R) \right).$$ Note that we may choose $g_P \in k(E)$ since the divisor is invariant under the action of the absolute Galois group of $k$. Now $\divisor\left(g_P^d\right) = \divisor\left( [d]^* t_P\right)$ and thus we may assume that $g_P^d = [d]^*t_P$. Similarly we find $g_Q \in k(E)$ with $g_Q^d = [d]^*t_Q$. \\
	Now consider the pullback of the torsor $\mathcal{T}$ along the generic point $\eta: \Spec k(E) \rightarrow \Spec k$. Fix a $\overline{k(E)}$-point $x_0$ of this pullback, i.e. a map of algebras so that $x_0 \circ [d]^* = \iota$, where $\iota: k(E) \rightarrow \overline{k(E)}$ is the inclusion. This means, that the following diagramms commute
	$$
	\xymatrix{
		& \Spec k(E) \ar[r]^-\eta \ar[d]^{[d]} & E \ar[d]^{[d]} \\
		\Spec \overline{k(E)} \ar[ru] \ar[r] & \Spec k(E) \ar[r]^-\eta & E}
	\qquad 	
	\xymatrix{ & k(E) \ar[ld]_{x_0} \\
		\overline{k(E)}& k(E)  \ar[u]_{[d]^*} \ar[l]^\iota}$$
	After possibly renaming $\alpha_P$ and $\alpha_Q$, we may assume that $x_0\left(g_P\right) = \alpha_P$ and $x_0\left(g_Q\right) = \alpha_Q$. 
	There is a group isomorphism
	$$	M \rightarrow \Gal\left(k(E)/[d]^*k(E) \right): R\mapsto \tau_R^*,$$
	where $\tau_R: E \rightarrow E$ is the translation by $R$-map; $\tau_R: E \rightarrow E: S \mapsto S \oplus R$. By the definition of the Weil-pairing $\e(R,P) = \frac{g_P (X\oplus S)}{g_P(X)} = \frac{ \tau_S^*(X)}{g_P(X)},$ for any $R \in M$, $X \in E(\kbar)$ any point so that $g_P(X)$ and $g_P(X \oplus S)$ are defined. The analogous result holds for $g_Q$ as well. Finally, we calculate   
	\begin{align*} 
	x_0 \circ \tau^*_{\left[\frac{\gamma(\alpha_Q)}{\alpha_Q}\right]_\rho P \ominus \left[\frac{\gamma(\alpha_P)}{\alpha_P}\right]_\rho Q} \left( g_P\right)
	&= x_0 \left( \e \left( \left[\frac{\gamma(\alpha_Q)}{\alpha_Q}\right]_\rho P \ominus \left[\frac{\gamma(\alpha_P)}{\alpha_P}\right]_\rho Q, P \right) g_P \right)\\
	&= x_0 \left( \e\left(\ominus \left[\frac{\gamma(\alpha_P)}{\alpha_P}\right]_\rho Q, P\right) g_P\right)\\
	&= x_0 \left({\frac{\gamma(\alpha_P)}{\alpha_P}}g_P \right)\\
	&= \gamma(\alpha_P)
	\end{align*}
	and similarly
	\begin{align*} 
	x_0 \circ \tau^*_{\left[\frac{\gamma(\alpha_Q)}{\alpha_Q}\right]_\rho P \ominus \left[\frac{\gamma(\alpha_P)}{\alpha_P}\right]_\rho Q} \left( g_Q \right)
	&= \gamma(\alpha_Q).	
	\end{align*}
	The statement follows since $k(E)/[d]^*k(E)$ is generated by $g_P$ and $g_Q$.
\end{proof}

Recall that $\epsilon$ is the composition 
\begin{equation*}
\mbox{\large{$\epsilon$ :}} \qquad 
\begin{gathered}
\xymatrix{
	H^1(k,M) \ar[r]^-\sim & H^1_{\text{\et}}\left(\Spec(k),M\right) \ar[r]^-{p^*}& H^1_{\text{\et}}(E,M)}\\
\xymatrix{ \ar[r]^-{- \cup [\mathcal{T}]}& H^2_{\text{\et}}(E,M \otimes M) \ar[r]^-\e& H^2_{\text{\et}}\left(E, \mu_d\right)  \ar[r]& \tor[d][\Br(E)]}.
\end{gathered}
\end{equation*}

In \cite[Theorem 4.1.1 and the following example]{skorobogatov}, the author proves abstractly that $\epsilon$ induces a splitting to (\ref{Eq:exact sequence on tor[d]Br(E)}) using general properties of torsors and the cup-product. We will now determine $\epsilon$ explicitly and prove directly that the map induces the desired splitting. 

\begin{prop}\label{Prop:epsilonk}On the level of cocycles $\epsilon$ coincides with the map that assigns to a 1-cocycle $f: G_k \rightarrow M$ the 2-cocycle 
	\begin{equation*}
	\begin{gathered}  
	\epsilon(f): G_{k(E)} \times G_{k(E)} \ra \overline{k(E)}^\times\\
	(\gamma,\tau) \mapsto \e\left( f(\gamma), \gamma \left( \left[\frac{\tau(\alpha_Q)}{\alpha_Q}\right]_\rho P \ominus \left[\frac{\tau(\alpha_P)}{\alpha_P}\right]_\rho  Q \right)  \right).
	\end{gathered} 
	\end{equation*} 
\end{prop}

\begin{proof}The cup-product commutes with $\eta^*$ by \cite[Chapter 2, 8.2]{Bredon:Sheaf-Theory}. Consider the following diagram with commutative squares
	$$
	\xymatrix{
		H^1(k,M) \ar[r]^{p^*} & H^1_{\text{\et}}(E,M) \ar[d]^{\eta^*} \ar[rr]^-{-\cup [\mathcal{T}]} && H^2_{\text{\et}}(E,M \otimes M) \ar[r]^-\e_-\sim \ar[d]^{\eta^*}	& H^2_{\text{\et}}(E,\mu_d)\ar[d] \\ 
		& H^1_{\text{\et}}(k(E),M) \ar[rr]^-{-\cup [\eta^*\mathcal{T}]} && H^2_{\text{\et}}(k(E), M \otimes M ) \ar[r]^-\e_-\sim & \tor[d][\Br k(E)]
	}$$
	The statement follows from \Cref{prop:torsormult3} and by the definition of the cup-product in group cohomology \cite[Chapter VIII, Section 3]{serrelocal}. Recall that by \cite[Theorem 5.11]{colliot-thelene-sansuc:therationalityproblem} the map on the right is given by the injection that identifies $\Br(E)$ with the unramified Brauer group of $k(E)$. 
\end{proof}

We are now ready to compute the elements of $\tor[d][\Br(E)]$. Since we assume that $M$ is $k$-rational, Kummer theory implies that there is an isomorphism 
\begin{align}\label{Eq:phi definition}  \phi: k^\times/(k^\times)^d \times k^\times/(k^\times)^d \rightarrow H^1(k,M): (a,b) \mapsto c_{a,b},\end{align} 
where $c_{a,b}$ can be represented by the cocycle
$$G_k \rightarrow M: \gamma \mapsto \left[ \frac{\gamma\left( \sqrt[d]{a} \right) }{\sqrt[d]{a}} \right]_\rho P \oplus \left[ \frac{\gamma\left( \sqrt[d]{b} \right) }{\sqrt[d]{b}} \right]_\rho Q.$$ 

\begin{prop}\label{Prop:epsilon-split-case}The cocycle
	$\epsilon \circ \phi (a,b)$ corresponds to the Brauer class of $$\left(a,t_P\right)_{d,k(E)} \otimes \left( b, t_Q \right)_{d,k(E)}$$ for any 
	$(a,b) \in k^\times/(k^\times)^d \times k^\times/(k^\times)^d$. 
\end{prop}

\begin{proof}Observe that  $	\epsilon_k\circ \phi(a,b)$ can be represented by the cocycle that takes $(\gamma,\tau) \in G_{k(E)} \times G_{k(E)}$ to 
	\begin{align*}
	&\e\left( \left[ \frac{\gamma\left( \sqrt[d]{a} \right) }{\sqrt[d]{a}} \right]_\rho P \oplus \left[ \frac{\gamma\left( \sqrt[d]{b} \right) }{\sqrt[d]{b}} \right]_\rho Q, \left[\frac{\tau(\alpha_Q)}{\alpha_Q}\right]_\rho P \ominus \left[\frac{\tau(\alpha_P)}{\alpha_P}\right]_\rho  Q  \right)\\
	=& \e\left( \left[ \frac{\gamma\left( \sqrt[d]{a} \right) }{\sqrt[d]{a}} \right]_\rho P, \left[\frac{\tau(\alpha_P)}{\alpha_P}\right]_\rho  Q  \right)^{-1} \e\left(\left[ \frac{\gamma\left( \sqrt[d]{b} \right) }{\sqrt[d]{b}} \right]_\rho Q, \left[\frac{\tau(\alpha_Q)}{\alpha_Q}\right]_\rho P \right).
	\end{align*}
	The statement follows from \Cref{prop:sumbolalgebra-cocycle}.
\end{proof}
Recall that we need to prove that $\epsilon$ induces a splitting to the sequence (\ref{Eq:exact sequence on tor[d]Br(E)}) on the right, i.e. we need to show that $r \circ \epsilon = \lambda$ and $\epsilon(\ker (\lambda))= 0$. We first describe $r$ explicitly. Recall from \cite{Lichtenbaum1969} the following commutative diagram with exact rows and columns
$$
\xymatrix{ & 0\ar[d] & 0 \ar[d]\\
\Br(k) \ar[r] \ar@{=}[d] & \Br(E)\ar[d]^\theta  & \H^1\left( k, \Pic(\overline{E})\right) \ar[d]^\nu\\
\Br(k) \ar[r] & \H^2\left( k, \overline{k}(E)^\times \right) \ar[r]^\mu \ar[d] & \H^2\left( k, \Prin(\overline{E})\right)\ar[d] \ar[r] & \H^3\left( k, \overline{k}^\times\right) \\
& \H^2 \left( k, \Div(\overline{E})\right) \ar[r]^\sim & \H^2\left( k, \Div(\overline{E})\right) } 
$$
where $\Prin(\overline{E})$ denotes the set of principal divisors, $\Div(\overline{E})$ the set of divisors, and $\Pic(\overline{E})$ the Picard group of $\overline{E}$. The right column is induced by the exact sequence 
$$0 \ra \Prin(\overline{E}) \ra \Div(\overline{E}) \ra \Pic(\overline{E}) \ra 0$$
and the second row is induced by 
$$0 \ra \kbar^\times \ra \kbar(E)^\times \ra \Prin(\overline{E}) \ra 0.$$
Therefore, for any $\alpha \in \Br(E)$, there is some unique element $\alpha' \in \H^1\left( k, \Pic(\overline{E})\right)$ with $\nu(\alpha') = \mu\circ \theta(\alpha)$. We will show in the following that $\alpha'$ is in the image of the embedding 
$$\H^1\left( k,E(\kbar)\right) \cong \H^1\left( k, \Pic^0(\overline{E})\right) \ra \H^1\left( k, \Pic(\overline{E})\right),$$
where $\Pic^0(\overline{E})$ denotes the degree $0$ piece of $\Pic(\overline{E})$. Consider the degree sequence
$$
0 \ra \Div^0(\overline{E}) \ra \Div(\overline{E}) \ra \ZZ \ra 0,
$$
where $\Div^0(\overline{E})$ is the group of degree zero divisors. Since $\H^1\left( k,\ZZ\right) = 0$, the map 
$$\H^2\left( k, \Div^0(\overline{E}) \right) \ra \H^2\left( k, \Div(\overline{E})\right)$$
is injective. This implies that $\mu \circ \theta (\alpha)$ is in the kernel of  
$$\H^2\left(k, \Prin(\overline{E})\right) \ra \H^2\left( k, \Div^0(\overline{E})\right)$$
induced by the short exact sequence 
$$0 \ra \Prin(\overline{E}) \ra \Div^0(\overline{E}) \ra \Pic^0(\overline{E}) \cong E(\kbar) \ra 0.$$
In particular, $\mu \circ \theta (\alpha)$ lifts uniquely to an element $\alpha'' \in \H^1\left( k, E(\kbar)\right)$. We then find that the map $r$, induced by the Hochschild-Serre spectral sequence, is given by $r(\alpha) =\alpha''$ \cite{Faddeev:SimpleAlgebrasOverAFieldOfAglebraicFunctionsOfOneVariable, Lichtenbaum1969}.

\begin{prop}\label{prop:kappcircepsilon=lambda_split}
	$r \circ \epsilon = \lambda$. 
\end{prop} 

\begin{proof}We will only prove that $r \circ \epsilon\circ \phi(a,1) = \lambda\circ \phi(a,1).$ The other cases are similar. We showed previously that $\epsilon \circ \phi(a,1) = (a,t_P)_{d,k(E)}$, which corresponds to the cocycle 
	$$(\gamma,\tau) \mapsto \begin{cases} 1 & \left[\frac{\gamma(\sqrt[d]{a})}{\sqrt[d]{a}}\right]_\rho^\Z + \left[\frac{\tau(\sqrt[d]{a})}{\sqrt[d]{a}}\right]_\rho^\Z < d \\
	t_P  & \text{else}\end{cases}$$
	in $H^2(k, \kbar(E)^\times)$. This gives an element in $H^2(k, \Prin(\overline{E}))$ via 
	$$(\gamma,\tau) \mapsto \begin{cases} 1 & \left[\frac{\gamma(\sqrt[d]{a})}{\sqrt[d]{a}}\right]_\rho^\Z + \left[\frac{\tau(\sqrt[d]{a})}{\sqrt[d]{a}}\right]_\rho^\Z < d \\
	d(P) - d(0)  & \text{else} \end{cases}.$$
	On the other hand for any $\gamma \in G_k$
	$$\lambda(\phi(a,1)) (\gamma) = \phi(a,1)(\gamma) =  \left[\frac{\gamma(\sqrt[d]{a})}{\sqrt[d]{a}}\right]_\rho^\Z P.$$
	Now we follow the proof of the snake lemma to calculate the image of $\lambda(f)$ under the connecting homomorphism 
	$$H^1(k, E(\kbar)) \rightarrow H^2(k, \Prin(\overline{E}))$$
	induced by the sequence 
	$$0 \ra \Prin(\overline{E}) \ra \Div^0(\overline{E}) \ra \Pic^0(\overline{E}) \cong  E(\kbar) \ra 0. $$
	First lift it to 
	$$\gamma \mapsto \left[\frac{\gamma(\sqrt[d]{a})}{\sqrt[d]{a}}\right]_\rho^\Z \left(  (P) - (0) \right) \in \Div^0(\overline{E}).$$
	Now use the boundary map to get\\
	\resizebox{.9\linewidth}{!}{
	\begin{minipage}{\linewidth}
	\begin{align*} 
	(\gamma,\tau) \mapsto &\gamma \left( \left[ \frac{\tau(\sqrt[d]{a})}{\sqrt[d]{a}}\right]_\rho^\Z \left(  (P) - (0) \right)\right) -  \left[ \frac{\gamma\tau(\sqrt[d]{a})}{\sqrt[d]{a}}\right]_\rho^\Z \left(  (P) - (0) \right) +  \left[\frac{\tau(\sqrt[d]{a})}{\sqrt[d]{a}}\right]_\rho^\Z \left(  (P) - (0) \right)\\
	&= \left[\frac{\tau(\sqrt[d]{a})}{\sqrt[d]{a}}\right]_\rho^\Z \left(  (P) - (0) \right) - \left[ \frac{\gamma\tau(\sqrt[d]{a})}{\sqrt[d]{a}}\right]_\rho^\Z \left(  (P) - (0) \right) + \left[ \frac{\tau(\sqrt[d]{a})}{\sqrt[d]{a}} \right]_\rho^\Z \left(  (P) - (0) \right)\\
	&=  \begin{cases} d(P) - d(0) & \text{if } \left[\frac{\gamma(\sqrt[d]{a})}{\sqrt[d]{a}}\right]_\rho^\Z + \left[\frac{\tau(\sqrt[d]{a})}{\sqrt[d]{a}}\right]_\rho^\Z \geq d \\
	1 & \text{else} \end{cases},
	\end{align*}\\
	\end{minipage}}\\
	which coincides with what we previously calculated. The statement follows.  
\end{proof} 

\begin{prop}\label{prop:epsilonkerlambda=0} 
	$\epsilon\left( \ker(\lambda) \right) = 0.$ 
\end{prop}
\begin{proof} 
	Recall that $\ker(\lambda) = \text{Im}(\delta)$ and let $R \in E(k)$. By the previous proposition $r\circ\epsilon\circ \delta (R) = \lambda \circ \delta(R)$ is trivial. Thus the algebra $\epsilon\circ \delta (R) $ is in the image of the $\Br(k) \rightarrow \Br(E)$. It remains to show that the specialization of $\epsilon \circ \delta(R)$ at $0$ is trivial. The cup-product commutes with specialization at a closed point \cite[Chapter 2, 8.2]{Bredon:Sheaf-Theory}, i.e. $\left( [\mathcal{T}'] \cup [\mathcal{T}]\right)_S = \left( [\mathcal{T}']\right)_S \cup \left( [\mathcal{T}]\right)_S$ for every $S \in E(k)$ and every $[\mathcal{T}']\in H^1_{\text{\et}} (E, M)$. By definition of $\epsilon$ $$\left( \epsilon\circ \delta (R) \right)_S =  \left( \epsilon\circ \delta (R) \right)_S =   \delta(R) \cup \mathcal{T}_S \in \Br(k)$$ for any $S \in E(k)$. 
	In particular, $\left( \epsilon\circ \delta (R) \right)_0 = \delta(R) \cup \mathcal{T}_0$. The specialization of $\mathcal{T}$ at $0$ admits a point (the point $0$) and is therefore the trivial torsor. We deduce that  $\left( \epsilon\circ \delta (R) \right)_0 $ is trivial and thus so is $\epsilon\circ \delta (R)$. 
\end{proof} 

We conclude that under the above assumptions, the $d$-torsion of $\Br(E)$ decomposes as 
	$$\tor[d][\Br(E)] =\tor[d][\Br(k)] \oplus I$$ and every element in $I$ can be represented as a tensor product 
	$$\left( a, t_P \right)_{d,k(E)} \otimes \left( b, t_Q \right)_{d,k(E)}$$ with $a,b \in k^\times$.
	We now proceed to determine the relations in $I$.  
Recall that an element in $I$ as before is trivial if and only if it is in the image of the composition 
$$\xymatrix{
	E(k)/[d]E(k) \ar[r]^-\delta & H^1(k,M) \ar[r]^-\epsilon & \tor[d][\Br(k)]
}.$$
We first need to describe the image of $\phi^{-1} \circ \delta$ explicitly, where $\phi$ is the isomorphism from (\ref{Eq:phi definition}).
\begin{prop}Let $R \in E(k)/[d]E(k)$ and let $t_P, t_Q \in k(E)$ with $\divisor(t_P) = d(P) - d(0)$, $\divisor (t_Q) = d(Q) -d(0)$. By the argument in the proof of \Cref{Prop:torsor mult by 3 over kbar}, we may assume that  $t_P\circ [d], t_Q \circ [d] \in (k(E)^\times)^d$. Then 
	$$\phi^{-1} \circ \delta (R) = \begin{cases} 
	(1,1) & R = 0\\
	\left( t_Q(P), \frac{t_P(P\oplus Q)}{t_P(Q)} \right) & R = P \\
	\left( \frac{ t_Q(P\oplus Q)}{t_Q(P)}, t_P(Q)  \right) & R= Q \\ 
	\left( t_Q (R), t_P(R) \right) & \text{else} \end{cases} $$
\end{prop}	

The proof of this proposition is inspired by a computation of the Kummer pairing in \cite[Ch. X, Theorem 1.1]{silverman}.  
\begin{proof} 
	Let $R \in E(k)/dE(k)$ and suppose that $R \neq 0,P,Q$. Fix some $S \in E(\kbar)$ with $[d]S = R$. Let $t_P, t_Q$ as above and fix $g_P, g_Q \in \kbar(E)$ with $g_P^d = t_P \circ [d]$ and $g_Q^d = t_Q \circ [d]$. Since the divisors of $g_P$ and $g_Q$ are $G_k$-invariant, we may choose $g_P, g_Q \in k(E)$.
	By the definition of $\phi$ we see that $\phi(f)=(a,b)$ for some cocycle $f$ representing a class in $ H^1(k,M)$ means exactly that 
	$\e\left(f(\gamma),P\right) = \frac{\gamma\left( \sqrt[d]{b}\right) }{\sqrt[d]{b}} \text{ and } \e\left(f(\gamma),Q\right) = \frac{\gamma\left( \sqrt[d]{a}\right) }{\sqrt[d]{a}}.$
	The Weil pairing satisfies 
	$$\e(\gamma(S) \ominus S,P ) = \frac{ g_P(\gamma(S) \ominus S \oplus S)}{g_P(S)}=  \frac{ g_P(\gamma(S))}{g_P(S)} =\frac{\gamma( g_P(S))}{g_P(S)}.$$
	Additionally by definition of $g_P$, we see that $g_P(S)^d = t_P \circ [d] (S) = t_P(R)$. A similar result holds for $Q$ as well. Therefore $\phi^{-1}\circ \delta(R) = 
	\left( t_Q (R), t_P(R) \right)$. The other results follow by linearity of the Weil pairing. 
\end{proof} 

Summarizing these results, we conclude \Cref{Prop:main result in split case}. 

\section{$M$ is not $k$-rational}\label{Mnotkrational}

Now let $k$ be any field and assume that $d=q$ is an odd prime coprime to the characteristic of $k$. Assume that $k$ contains a primitive $q$-th root of unity $\rho$. Let $E$ be an elliptic curve over $k$ and denote its $q$-torsion by $M$. We do not assume that $M$ is $k$-rational throughout this section. Consider the Galois representation 
$$\Psi: G_k \rightarrow GL_2(\mathbb{F}_q)$$ given by the action of $G_k$ on $M$ and denote the fixed field of its kernel by $L$. Since $k$ contains a primitive $q$-th root of unity $\rho$, the image of $\Psi$ lies in $SL_2(\mathbb{F}_q)$: For $\sigma \in G_k$ with $\Psi(\sigma) = \begin{pmatrix} a&b\\c&d \end{pmatrix}$, we have that 
$$\begin{gathered}
\rho = \sigma(\rho) = \sigma \e (P,Q) = \e \left(\sigma(P), \sigma(Q) \right) = \e \left( aP \oplus cQ, bP \oplus dQ \right) \\
= \e(P,Q)^{ad-bc} = \rho^{ad-bc}.\end{gathered}$$
Therefore $ad-bc=1$ and $\Psi(\sigma) \in SL_2(\mathbb{F}_q)$. 

Consider the tower of field extensions 
$$\xymatrix{
	& \overline{k(E)}\ar@{-}[d]\\
	& L(E) \ar@{-}[ld]_M \ar@{-}[rd]^{\Gal(L/k)}\\
	[q]^*L(E)\ar@{-}[rd]^{\Gal(L/k)} & & k(E) \ar@{-}[ld]\\
	& [q]^* k(E)
}$$
Fix a set $\tilde{G}_{L/k} \subset G_{k(E)}$ of coset representatives of $G_{k(E)} / G_{L(E)} \cong \Gal(L/k)$. Note that $\tilde{G}_{L/k}$ is also a set of coset representatives of 
$G_{[q]^*k(E)} / G_{[q]^*L(E)} \cong \Gal(L/k)$ and every $\tilde{\sigma} \in \tilde{G}_{L/k}$ fixes $k(E)$. Let $\gamma \in G_{[q]^* k(E)}$, then $\gamma$ decomposes as $\gamma= \gamma' \tilde{\sigma}$ for some $\gamma' \in G_{[q]^*L(E)}$ and some $\tilde{\sigma} \in \tilde{G}_{L/k}$. \\

Let $\mathcal{T}$ be the torsor given by multiplication by $q$ on $E$. Fix a set of coset representatives $\tilde{G}_{L/k}$ as before. We now describe the cocycle corresponding to the pullback of $\mathcal{T}$ along the generic point using the correspondence in  \Cref{ch:background}.
Let $x_1$ be a $\overline{k(E)}$ point. We may assume, that $x_0 = \iota_1 \circ x_0$ for some $x_0$ as in the following commutative diagram. 
$$
\xymatrix{
	k(E) \ar[r]^{\iota_1}\ar@/^3pc/[drr]^{x_1} & L(E) \ar[rd]^{x_0} \\ 
	k(E) \ar[u]^{[q]^*} \ar[r]^{\iota_1}  &L(E) \ar[u]^{[q]^*} \ar[r]^\iota & \overline{k(E)} }.
$$
Any $\gamma = \gamma' \tilde{\sigma} \in G_{k(E)}$ with $\gamma' \in G_{L(E)}, \tilde{\sigma} \in \tilde{G}_{L/k}$ acts on $x_1$ by 
$$\gamma.x_1 = \gamma' \circ \tilde{\sigma} x_0 \circ \iota_1 =  \gamma'. x_1.$$
Finally, $\gamma'.x_1$ can be computed as in \Cref{Prop:torsor mult by 3 over kbar}.
Summarizing this we conclude the following proposition. 

\begin{prop}\label{prop:torsormult3}
	The pull-back of $\mathcal{T}$ to the generic point corresponds to the element in $H^1(k(E),M)$ given by the cocycle 
	$$G_{k(E)} \ra M: \gamma \mapsto \left[\frac{\gamma'(\alpha_Q)}{\alpha_Q}\right]_\rho P \ominus \left[ \frac{\gamma'(\alpha_P)}{\alpha_P}\right]_\rho Q,$$
	where $\gamma = \gamma' \tilde{\sigma}$ as above, for some $\tilde{\sigma} \in \tilde{G}_{L/k}$ and $\gamma' \in G_{L(E)}$, $\alpha_P, \alpha_Q\in \overline{k(E)}$ so that $\alpha_P^q = t_P$ and $\alpha_Q^q = t_Q$.  
\end{prop}

As in the previous section, we describe the map $\epsilon$ explicitly. 

\begin{prop}On the level of cocycles $\epsilon$ coincides with the map that assigns to a 1-cocycle $f: G_k \rightarrow M$ the 2-cocycle 
	\begin{equation*}
	\begin{gathered}  
	\epsilon(f): G_{k(E)} \times G_{k(E)} \ra \overline{k(E)}^\times\\
	(\gamma,\tau) \mapsto \e\left( f(\gamma), \gamma \left( \left[\frac{\tau'(\alpha_Q)}{\alpha_Q}\right]_\rho P \ominus \left[\frac{\tau'(\alpha_P)}{\alpha_P}\right]_\rho  Q \right)  \right),
	\end{gathered} 
	\end{equation*} 
	for $\gamma,\tau \in G_{k(E)}, \tau = \tau' \tilde{\sigma}$ for some $\sigma \in \tilde{G}_{L/k}$ and $\tau' \in G_{\kbar(E)}$. 
\end{prop}

\subsection{$[L:k]$ is coprime to $q$ }\label{sec:Lkcoprimetoq}

Suppose throughout this section that $q$ does not divide the order $[L:k]$. 
Let $\mathcal{T}$ be the torsor given by multiplication by $d$ on $E$. Consider the pull-back $\eta_L^*(\mathcal{T})$ of $\mathcal{T}$ to $L(E)$ 
and the pull-back $\eta_k^*(\mathcal{T})$ of $\mathcal{T}$ to $k(E)$. By \Cref{Prop:torsor mult by 3 over kbar} and \Cref{prop:torsormult3} we see immediately that $\res\left( \eta^*_k (\mathcal{T}) \right) = \eta^*_L (\mathcal{T})$. By \cite[Ch. 1, Proposition 1.5.3 (iii) and (iv)]{Neukirch-Schmidt-Wingberg-Cohomology-of-number-fields} and the construction of $\epsilon$ the following diagram commutes
$$\xymatrix{
	H^1(k,M) \ar[r]^\res \ar[d]^{\epsilon_k} & H^1(L,M) \ar[r]^\Cor \ar[d]^{\epsilon_L} & H^1(k,M) \ar[d]^{\epsilon_k}\\
	\tor[q][\Br(E)] \ar[r]^\res & \tor[q][\Br(E \otimes L)] \ar[r]^\Cor & \tor[q][\Br(E)]
}.$$
The corestriction maps in the above diagram are both surjective, so that  every element in $I$ can be written as $\Cor(A)$ with $A \in I_L$, where $I_L$ is the image of the map $\epsilon_L$. We summarize this observation in the following theorem. 
\begin{thm}
	Let $t_P,t_Q \in L(E)$ with divisors  $\divisor(t_P) = q(Q) - q(0)$ and $\divisor(t_Q) = q(Q)- q(0)$. Then the $q$-torsion of $\Br(E)$ decomposes as 
	$$\tor[q][\Br(E)] =\tor[q][\Br(k)] \oplus I$$ and every element in $I$ can be represented as a tensor product 
	$$\Cor\left( a, t_P \right)_{q,L(E)} \otimes \Cor\left( b, t_Q \right)_{q,L(E)}$$ with $a,b \in L^\times$. 
\end{thm}

\begin{rem}\label{rem:image of restriction} The corestriction map is in general not injective. To get a smaller set, remark that by \cite[Ch, 1, Corollary 1.5.7]{Neukirch-Schmidt-Wingberg-Cohomology-of-number-fields} the image of the restriction map $H^1(k,M) \ra H^1(L,M)$ coincides with the image of the norm map $$N_{L/k}: H^1(L,M) \ra H^1(L,M).$$ Now by Kummer theory $H^1(L,M) \cong L^\times/(L^\times)^q \times L^\times/(L^\times)^q$ via the isomorphism $\phi$. Let $g \in G_k$ and $(a,b) \in L^\times/(L^\times)^q \times L^\times/(L^\times)^q$. Suppose that $g^{-1}(P) = c_1 P \oplus c_2 Q$ and $g^{-1}(Q) = c_3P \oplus c_4Q$. The action of $g$ compatible with $\phi$ is  
	\begin{align*}
	g.(a,b) =& \phi^{-1} g. \phi(a,b) \\
	=& \left( \left(  g^{-1}(a) \right)^{c_1} \left( g^{-1}(b) \right)^{c_3}, \left(  g^{-1}(a) \right)^{c_2} \left( g^{-1}(b) \right)^{c_4}\right) .
	\end{align*}
	Now the image of the restriction followed by $\phi$ coincides with the image of the norm on $L^\times/(L^\times)^q \times L^\times/(L^\times)^q$ under the above action. 
\end{rem} 
To calculate the relations consider the following commutative diagram 
$$\xymatrix{
	E(k)/[q]E(k) \ar[r]^{\res} \ar[d]^{\delta_k}& E(L)/[q]E(L) \ar[r]^{\Cor} \ar[d]^{\delta_L}& E(k)/[q]E(k) \ar[d]^{\delta_k}\\
	H^1(k,M) \ar[r]^{\res} & H^1(L,M) \ar[r]^{\Cor} & H^1(k,M) 
},$$where the horizontal compositions coincide with multiplication by $[L:k]$, which is an isomorphism. Thus, the image of $\delta_k$ is also given by the image of the composition $\delta_k \circ \Cor = \Cor \circ \delta_L$. Using the description of the image of $\delta_L$ in the previous section we deduce the following result. 

\begin{prop} 
	Suppose that $[L:k]$ is not divisible by $q$. Fix two generators $P$ and $Q$ of $M$ and and let $t_P, t_Q \in L(E)$ with $\divisor(t_P) = q(P) - q(0)$ and $\divisor(t_Q) = q(Q) - q(0)$. Assume additionally that $t_P \circ [q], t_Q \circ[q] \in \left( L(E)^\times \right)^d.$ An element 
	$$\Cor(a,t_P)_{L(E)} \otimes \Cor(b,t_Q)_{L(E)}$$in $I$ is trivial if $a$ and $b$ satisfy one of the following in $L^\times/(L^\times)^q$:
	\begin{itemize}
	    \item $a \equiv 1$ and $b\equiv 1$,
		\item $a \equiv t_Q(P)$ and $b\equiv \frac{t_P(P\oplus Q)}{t_P(Q)}$,
		\item $a \equiv \frac{ t_Q(P \oplus Q)}{t_Q(P)}$ and $b= t_P(Q)$, or  
		\item $a \equiv t_Q (R)$ and $b\equiv t_P(R)$ for some $R \in E(k)\setminus\{0,P,Q\}$. 
	\end{itemize}
	Remark that there may be additional relations in $I$, as the corestriction map is not injective.
\end{prop} 

The following observation will be useful to calculate these corestrictions explicitly. Consider the following commutative diagram
\begin{equation}\small\label{eqn:relationscorrescorres} 
\xymatrix{
	E(k) \ar[r] \ar[d]^{\delta_k}& E(L) \ar[r] \ar[d]^{\delta_L} & E(k) \ar[r] \ar[d]^{\delta_k}& E(L) \ar[r] \ar[d]^{\delta_L} & E(k) \ar[d]^{\delta_k}\\
	H^1(k,M) \ar[r]^\res \ar[d]^{\epsilon_k}& H^1(L,M) \ar[r]^\Cor \ar[d]^{\epsilon_L}& H^1(k,M)\textbf{} \ar[r]^\res\ar[d]^{\epsilon_k}  & H^1(L,M) \ar[r]^\Cor \ar[d]^{\epsilon_L}& H^1(k,M) \ar[d]^{\epsilon_k}\\
	\tor[q][\Br E] \ar[r] & \tor[q][\Br E_L] \ar[r] &  \tor[q][\Br E] \ar[r] & \tor[q][\Br E_L] \ar[r] &  \tor[q][\Br E]
},\end{equation} 
where the rows compose to multiplication by $[L:k]^2$, which is an isomorphism. Furthermore, the composition $\res \circ \Cor$ coincides with the norm map \cite[Ch. 1, Corollary 1.5.7]{Neukirch-Schmidt-Wingberg-Cohomology-of-number-fields}. Therefore, an element in $I$ is trivial if it lies in the image of the composition $\Cor\circ \epsilon_L \circ N_{L/k} \circ \delta_L \circ \res$. In \Cref{subsec:Ex degree L/k coprime to 3}, we see how this observation can be applied to the calculation of the relations in $I$. 

\subsection{$[L:k]$ equals $q$}\label{sec:Generators-div-q} 

Suppose for this section that $L$ is of degree $q$ over $k$. After renaming $P$ and $Q$ we may assume without loss of generality that there is some $\sigma \in G_k$ such that $\sigma(Q) = P\oplus Q$ and $\overline{\sigma}$ generates $G_k/G_L$. Replacing $P$ by $nP$ and $Q$ by $nQ$ for $n = \left[\e(Q,P)\right]_\rho$, we may still assume that $e(P,Q) = \rho$. Fix a coset representative $\tilde{\sigma}$ of $\sigma$ in $G_{k(E)}$. 
Additionally denote a primitive element for the extension $L/k$ by $l$.
Consider the diagram
$$\xymatrix{ 0 \ar[r]& H^1\left(G_k/G_L, M\right) \ar[r]^-{\inflation}& H^1(k,M)\ar[rr]^-\res\ar[d]^{\epsilon_k} && H^1(L,M)^{G_k/G_L}\ar[d]^{\epsilon_L}\\
	& & \tor[q][\Br(E)]\ar[rr]^-\res && \tor[q][\Br(E \times_k \Spec L)] },$$
where the first row is the inflation-restriction exact sequence. The diagram commutes by construction of $\epsilon$ and since the restriction map and the cup-product commute \cite[Ch. 1 Proposition 1.5.3 (iii)]{Neukirch-Schmidt-Wingberg-Cohomology-of-number-fields}. We will first describe the image of the inflation map and explore the restriction afterwards. We will apply the following technical lemma throughout. 

\begin{lem}\label{lem:technical}
	$\sum_{i=0}^{q-1} \sigma^i(R) = 0$ for every $R \in M$.
\end{lem} 

\begin{proof} 
	Let $R= mP\oplus nQ \in M$. We calculate directly that
	$$
	\sum_{i=0}^{q-1} \sigma^i(mP \oplus nQ) =  \sum_{i=0}^{q-1} \left( mP \oplus inP\oplus nQ \right) = mqP \oplus \frac{q (q-1)}2 n P \oplus n q Q = 0. \qedhere
	$$  
\end{proof}

\begin{lem} The group
	$H^1\left( G_k/G_L, M \right)$ is cyclic of rank $q$ with generator the class of the cocycle $f_L$ defined by $f_L(\overline{\sigma}) = Q$. 
\end{lem} 

\begin{proof} 
	\Cref{lem:technical} implies that  $f_L(\overline{\sigma}^q) = \sum_{i=0}^{q-1} \sigma^i f_L(\overline{\sigma}) = 0$ and thus $f_L$ defines a cocycle. Since $G_k/G_L$ is cyclic with generator $\overline{\sigma}$, every element $f$ in $H^1(G_k,G_L,M)$ is determined by $f(\overline{\sigma})$. Furthermore, if $f(\overline{\sigma}) = mP\oplus nQ$, then 
	\begin{align*}
	f(\overline{\sigma}) - \overline{\sigma}(mP) = mP \oplus nQ \ominus mP =  nQ = f_L^n (\overline{\sigma}).
	\end{align*} 
	The statement follows. 
\end{proof} 
Let $\alpha_Q \in \overline{k(E)}$ with $\alpha_Q^q = t_Q$ and consider $n_Q = \prod_{i=0}^{q-1} \tilde{\sigma}^i \left(\alpha_Q\right)$. Fix $\gamma \in G_{L(E)}$. By our previous calculations and with $x_0$ and $g_Q$ as in the proof of \Cref{prop:torsormult3} we deduce that there is some $R \in M$ such that $\gamma(\alpha_Q) = R.x_0(g_Q)$. Then 
by \Cref{lem:technical} 
\begin{align}
\gamma \left( \prod_{i=0}^{q-1} \tilde{\sigma}^i \left(\alpha_Q\right) \right)  = \left( \sum_{i=0}^{p-1} \sigma^i(R) \right). x_0(g_Q) = \alpha_Q.
\end{align} Now $\prod_{i=0}^{q-1} \tilde{\sigma}^i \left(\alpha_Q \right) $ is obviously fixed by $\tilde{\sigma}$ and therefore $\prod_{i=0}^{q-1} \tilde{\sigma}^i \left(\alpha_Q\right) \in k(E).$ Thus $n_Q$ is defined to be the element in $k(E)$ with $n_Q^q = N_{L(E)/k(E)} (t_Q)$. Note additionally that $\divisor n_Q = \sum_{i=0}^{q-1} \left(\sigma^i(Q) - (0)\right)$.

\begin{prop}\label{prop:ImageofInflation}
	$\epsilon_k\left(\inflation\left(f_L\right)\right)$ is the inverse of the Brauer class of the symbol algebra
	$\left(l^q, n_Q\right)_{q,k(E)},$ where $\alpha_Q \in \overline{k(E)}$ with $\alpha_Q^q = t_Q$.
\end{prop} 

\begin{proof} Let $\gamma, \tau \in G_{k(E)}$ and denote $\gamma = \gamma' \tilde{\sigma}^i, \tau = \tau' \tilde{\sigma}^j$ with $\gamma',\tau' \in G_{L(E)}$. Then by definition of $\epsilon$ we see that  
	\begin{align*} 
	\epsilon_k \left( \inflation\left( f_L \right) \right) (\gamma,\tau) 
	&=\e \left( \frac{(i-1) i }{2} P \oplus iQ , \sigma^i \left( \left[\frac{ \tau'\left( \alpha_Q \right)}{\alpha_Q}\right]_\rho P \ominus \left[\frac{ \tau' \left( \alpha_P \right)}{\alpha_P}\right]_\rho Q \right) \right)  \\
	&= \left( \frac{ \tau' \left( \alpha_P \right)}{\alpha_P} \right)^{- \frac{(i-1)i}{2}} \left(\frac{  \tau'(\alpha_Q)}{\alpha_Q} \right)^{-i} \left( \frac{ \tau' \left( \alpha_P \right)}{\alpha_P} \right)^{i^2}\\
	&= \left( \frac{ \tau' \left( \alpha_P \right)}{\alpha_P} \right)^{ \frac{(i+1)i}{2}} \left(\frac{  \tau'(\alpha_Q)}{\alpha_Q} \right)^{-i}\end{align*} 
	Now consider the map 
	\begin{equation}
	g: G_{k(E)} \rightarrow \overline{k(E)}^\times : \gamma \mapsto \gamma' \left( \prod_{n=0}^{i-1} \tilde{\sigma}^n(\alpha_Q) \right),
	\end{equation}
	where $\gamma= \gamma' \tilde{\sigma}^i$ for some $\gamma' \in G_{L(E)}$. The differential of $g$ is 
	$$dg(\gamma,\tau) = 
	\begin{cases} 
	\frac{ \prod_{n=0}^{i-1} \tilde{\sigma}^n(\alpha_Q) }{ \tilde{\sigma}^i \tau' \tilde{\sigma}^{-i} \left( \prod_{n=0}^{i-1} \tilde{\sigma}^n(\alpha_Q) \right)} & \text{ if } i + j < q\\
	\left(  \prod_{n=0}^{q-1} \tilde{\sigma}^n(\alpha_Q)\right) \left( \frac{ \alpha_P}{\tau'(\alpha_P)} \right)^{-\frac{(i+1)i}{2}}  \left(  \frac{ \alpha_Q}{\tau'(\alpha_Q)}\right)^i & \text{ else}
	\end{cases}. $$
	The statement follows by subtracting this trivial cocycle and applying \Cref{prop:sumbolalgebra-cocycle}.
\end{proof} 

\begin{prop} 
	$\epsilon$ induces a splitting of sequence (\ref{Eq:exact sequence on tor[d]Br(E)}).
\end{prop} 

\begin{proof}
	It suffices to show that $r \circ \epsilon \circ \inflation (f_L) = \lambda \circ \inflation (f_L)$. 
	By the previous lemma, $\epsilon(\inflation(f_L))$ gives
	$$(\gamma'\sigma^i,\tau'\sigma^j) \mapsto \begin{cases} 1 & \\ i+j < q \\
	\sum_{i=0}^{q-1} \left(\sigma(Q) - (0)\right) & i+j \geq q \end{cases}$$
	in $H^2(G_k, \Prin(E))$, where $\gamma',\tau' \in G_k$. On the other hand, $r(\inflation(f_L))$ can be presented by the cocycle 
	$$\gamma' {\sigma}^i \mapsto \sum_{m=0}^i {\sigma}^m(Q) \ominus 0.$$
	This lifts to the map
	$$\gamma' {\sigma}^i \mapsto \sum_{m=0}^i \left( \sigma^m(Q)\right)  - (0) \in \Div^0(\overline{E}),$$
	and a direct computation of the boundary map gives 
	\begin{align*} 
	(\gamma' {\sigma}^i,\tau'{\sigma}^j) \mapsto &
	\begin{cases}  \sigma^i \left( \sum_{m=0}^j \left( \sigma^m(Q)\right)  - (0)\right) - \left( \sum_{m=0}^{i+j} \left( \sigma^m(Q)\right)  - (0)\right) \\
	+ \left( \sum_{m=0}^i \left( \sigma^m(Q)\right)  - (0)\right) & i+ j < q \\
	\sigma^i \left( \sum_{m=0}^j \left( \sigma^m(Q)\right)  - (0)\right) - \left( \sum_{m=0}^{i+j-q} \left( \sigma^m(Q)\right)  - (0)\right) \\
	+ \left( \sum_{m=0}^i \left( \sigma^m(Q)\right)  - (0)\right) & i+ j \geq q \end{cases} \\
	&=  \begin{cases} 1 & i+ j < q \\
	\left( \sum_{m=0}^{q-1} \left( \sigma^m(Q)\right)  - (0)\right) & i+ j \geq q \end{cases},
	\end{align*}
	which coincides with the previous calculation.
\end{proof}

We now calculate the image of the composition $\phi^{-1} \circ \res$ with $\phi$ as in (\ref{Eq:phi definition}). By \cite[Chapter VII, Section 5]{serrelocal} the action of $G_k$ on $L^\times/(L^\times)^q \times L^\times/(L^\times)^q$ compatible with $\phi$ is given by 
$$\sigma.(a,b) = \phi^{-1} \sigma. \phi(a,b) = \left( \frac{\sigma^{-1}(a)}{ \sigma^{-1}(b)} , \sigma^{-1}(b) \right).$$  

\begin{lem}\label{lem:fixedSet}
	Under the isomorphism $\phi^{-1}$ the fixed set $H^1(L,M)^{G_k/G_L}$ corresponds to 
	$$\left\{ \left( a, \frac{a}{\sigma(a)} \right) : \sigma(a)^2 \equiv \sigma^2(a) a \mod (L^\times)^q \right\}.$$
\end{lem} 

\begin{proof} 
	Let $(a,b) \in L^\times/(L^\times)^q \times L^\times/(L^\times)^q$ be fixed by the above action. Then 
	$a \equiv \frac{\sigma^{-1}(a)}{ \sigma^{-1}(b)}$, which implies that $b \equiv \frac{a}{\sigma(a)}$. Now $b \equiv \sigma^{-1}(b)$ and thus 
	$ \frac{a}{\sigma(a)} \equiv \frac{ \sigma^{-1}(a)}{a}$ which implies that $a^2 \equiv \sigma(a) \sigma^{-1}(a)$ or equivalently $\sigma(a)^2 \equiv \sigma^2(a) a$. 
\end{proof} 

\begin{lem}\label{lem:ImageRes1} 
	$f \in H^1\left( L,M\right)^{G_k/G_L}$ is in the image of the restriction map if and only if 
	$f(\gamma^q) = 0$ for any $\gamma \in G_k$. 
\end{lem} 

\begin{proof}
	Let $\gamma \in G_k$ and suppose that $f$ is in the image of the restriction map with preimage $g$. Then we can write $\gamma = \gamma' \sigma^i$ for some $\gamma' \in G_L$. We calculate directly using \Cref{lem:technical} that
	$$f(\gamma^q) = g(\gamma^q) = \sum_{i=0}^{q-1} \gamma^q g(\gamma) = \sum_{i=0}^{q-1} \sigma^{iq} g(\gamma) = \sum_{i=0}^{q-1} \sigma^i  g(\gamma) = 0.$$
	For the converse assume that $f$ satisfies the condition that $f(\gamma^q) = 0$ for any $\gamma \in G_k$. In particular $f(\sigma^q) = 0$. Define $g \in H^1(k,M)$ by setting $g(\gamma) = f(\gamma')$, where $\gamma = \gamma' \sigma^i$ for $\gamma' \in G_L$. This is well-defined as for any $\gamma,\tau\in G_k$ with $\gamma = \gamma' \sigma^i, \tau = \tau' \sigma^j$ and $\gamma',\tau' \in G_L$ we have that 
	$
	g\left( \gamma\tau\right) = g\left( \gamma' \left( \sigma^i \tau' \sigma^{-i}\right)  \sigma^{i+j}\right)  
	= g\left( \gamma'\right) g\left( \sigma^i \tau' \sigma^{-i}\right)
	= g\left( \gamma' \right) \sigma^i g(\tau') 
	= g\left( \gamma \right) \gamma g(\tau).
	$\end{proof} 

We will now prove a technical lemma that will be useful to determine the image of the restriction. 

\begin{lem}\label{lem:technical-lemma-Galois-theory}
	Let $k$ be a field of characteristic prime to $q$ containing a primitive $q$-th root of unity. Let $k \subset L \subset F$ be a tower of field extensions so that each extension is Galois of degree $q$. Let $a \in L^\times$ such that $F = L\left( \sqrt[q]{a}\right)$. Fix a representative $\sigma \in G_k$ that generates $\Gal(L/k)$. Suppose that for every $\gamma \in G_k$ we have that $\gamma^q \left( \sqrt[q]{\sigma^i(a)}\right)= \sqrt[q]{\sigma^i(a)}$ for $0 \leq i < q$. Then there exists some $b \in k^\times$ such that $a \equiv b \mod (L^\times)^q$. 	
\end{lem} 

\begin{proof}Assume that $a \not\in k^\times$. Fix $\sigma \in G_k$ such that $\overline{\sigma}$ generates $\Gal(L/k)$. Denote the Galois closure of the extension $F$ over $k$ by $\tilde{L}$. By Galois theory 
	$\tilde{L} = L\left( \sqrt[q]{a}, \sqrt[q]{\sigma(a)}, \ldots, \sqrt[q]{\sigma^{q-1}(a)} \right)$ and $\Gal(\tilde{L}/L)$ is isomorphic to $\left( \Z/q\Z \right)^{r}$ for $r=1$, or $r=q$. 
	We prove the lemma by contradiction. Assume that $r=q$. Let $\tau \in \Gal(\tilde{L}/L)$ be the element with $\tau \left( \sqrt[q]{a} \right) = \rho \sqrt[q]{a}$ and $\tau \left( \sigma^i \sqrt[q]{a} \right) = \sigma^i \sqrt[q]{a}$ for $1 \leq i <q$. 
	Now $\left( \sigma\tau \right) \left( \sqrt[q]{a} \right)= \left( \rho \sigma^q\right) \left( \sqrt[q]{a} \right) =\rho \left(   \sqrt[q]{a}\right)$, which is a contradiction to our assumptions. We conclude that $F/k$ is Galois of degree $q^2$. We want to show that $\Gal(F/k) = \Z/q\Z \times \Z/q\Z$.	
	Next suppose that $\Gal(F/k) = \Z/q^2 \Z$ and denote the generator by $\tau$. Then $\tau^q$ fixes $L$, as $\tau^q(\sqrt[q]{a}) =\sqrt[q]{a}$ implies that $\tau^q(a) = a$ and $L= k(a)$. Hence $\tau \in \Gal(F/L)$, which implies that $\tau$ is of order $q$, a contradiction. We conclude that $\Gal(F/k) \cong \Z/q\Z \times \Z/q\Z$. Consider the fixed field $F^{\left< \sigma \right>}$, which is a degree $q$ extension of $k$. By Kummer theory, there exists an element $b \in k^\times$ so that  $F^{\left< \sigma \right>} = k\left( \sqrt[q]{b}\right)$. Finally, $F = L \left( \sqrt[q]{b}\right) = L \left( \sqrt[q]{a}\right)$ and thus by Kummer theory $a \equiv b \mod (L^\times)^q$. 
\end{proof} 

\begin{prop}\label{prop:ImageOfRestrictionmap}
	The image of the restriction map corresponds to the set $$k^\times/\left((L^\times)^q\cap k^\times\right) \times \{1\} \subset L^\times/(L^\times)^q \times L^\times/(L^\times)^q$$ under the isomorphism $\phi^{-1}$. 
\end{prop} 

\begin{proof} 
	Let $(a,b) \in L^\times/(L^\times)^q \times L^\times/(L^\times)^q$ be in the image of the restriction map. There exists some $f \in H^1(k,M)$ so that $\phi^{-1}\circ \res(f) = (a,b)$. Then $(a,b)$ is necessarily in the preimage $\phi^{-1} \left( H^1(L,M)^{G_k/G_L} \right) $ and by \Cref{lem:fixedSet} we see that $(a,b) \equiv \left( a, \frac{a}{\sigma(a)} \right)$ with $\sigma(a)^2 \equiv \sigma^2(a) a \mod (L^\times)^q$. It remains to show that we can choose $a \in k^\times$. \\
	By definition of $\phi$ and by \Cref{lem:ImageRes1} we get that $\gamma^q\left(\sqrt[q]{a}\right) = \sqrt[q]{a}$ and $\gamma^q\left( \sqrt[q]{\frac{a}{\sigma(a)}}\right) = \sqrt[q]{\frac{a}{\sigma(a)}}$ for any $\gamma \in G_k$ and for any choice of root. Using the condition that $\sigma(a)^2 \equiv \sigma^2(a) a$, we deduce that $\gamma^p \left( \sqrt[q]{\sigma^i}{a} \right) = \sqrt[q]{\sigma^i}{a}$ for any $i$. The statement follows from \Cref{lem:technical-lemma-Galois-theory}.	
\end{proof} 

The following theorem summarizes the results of this section. 

\begin{thm}\label{thm:Br(E) three case} 
	Suppose that $[G_k:G_L]$ is of order $q$ and choose $\sigma \in G_k$ with $\sigma(Q) = P\oplus Q$ such that $\overline{\sigma}$ generates $G_k/G_L$.  Additionally denote a primitive element for the extension $L/k$ by $l$.  
	$$\tor[q][\Br(E)] = \tor[q][\Br(k)] \oplus I$$
	and every element in $I$ can be written as 
	$$\left( \left( l^q\right)^i , n_Q \right)_{k(E)} \otimes \left( a,t_P\right)_{k(E)},$$
	for some $0 \leq i < q, a \in k^\times,$ and $n_Q \in k(E)$ with $n_Q^q = N_{L(E)/k(E)}(t_Q)$.
\end{thm} 
For the relations consider the commutative diagram with exact rows and columns 
$$\xymatrix{
	& & 0 \ar[d] & 0 \ar[d] \\
	0 \ar[r] & \frac{E(k) \cap [q]E(L)}{[q]E(k)} \ar[r]^{\inflation} \ar[d]^{\delta_{L/k}} & E(k)/[q]E(k) \ar[r]^{\res}\ar[d]^{\delta_k} & E(L)/[q]E(L) \ar[d]^{\delta_L}   \\
	0 \ar[r]& H^1(\Gal(L/k),M) \ar[r]^{\inflation} & H^1(k,M) \ar[r]^{\res} & H^1(L,M) 
},$$
where $\delta_{L/k}$ is the map induced by $\delta_k$. It is immediate that $\delta_{L/k}$ is injective. Recall that $H^1\left( \Gal(L/k),M\right)$ is cyclic of order $q$ with generator $f_L$. Furthermore, we saw that $\epsilon_L\left( \inflation(f_L) \right) = \left( l^q, n_Q \right)_{k(E)}$ (\Cref{prop:ImageofInflation}). We deduce the following results. 

\begin{enumerate} 
	\item The Brauer class of $\left( l^q, n_Q \right)_{k(E)} $ is trivial, that is it is in the image of the map $\Br(k) \rightarrow \Br(E)$, if and only if the quotient 
	$\frac{E(k) \cap [q]E(L)}{[q]E(k)} $ is non-trivial. 
	\item 	The Brauer class of $\left( a, t_P \right)_{k(E)}$ is trivial if and only if there is some $R \in E(k)/[q]E(k)$ so that $\phi^{-1}(a,1) = \delta_L (R)$. 
\end{enumerate}

\subsection{$q$ divides the degree $[L:k]$}\label{sec:Generators:qdividesLk} 

Suppose for this section that $q$ divides $[L:k]$. Let $k \subset L' \subset L$ be an intermediate field so that $L/L'$ is a Galois extension of degree $q$ and $q$ does not divide the degree $L'/k$. After renaming $P$ and $Q$ we may assume that there is some generator $\overline{\sigma}$ of $\Gal(L/L')$ so that $\sigma(P) = P$ and $\sigma(Q) = P\oplus Q$. Furthermore, let $l \in L$ with $l^q \in L'$ and $L = L'(l)$.
Fix $t_P, t_Q \in L(E)$ with $\divisor(t_P) = q(P) - q(0)$ and $\divisor(t_Q) = q(Q) - q(0)$. Assume additionally that $t_P \circ [q], t_Q \circ[q] \in \left( L(E)^\times \right)^d.$  Furthermore, let $\alpha_Q\in \overline{k(E)}$ so that $\alpha_Q^q = t_Q$. \\

Although the field extension $L'/k$ might not be Galois, restriction followed by corestriction coincides with multiplication by $[L':k]$, which is an isomorphism. Using the previous section we deduce the following result. 

\begin{prop} 
	Under the above assumptions, the Brauer group decomposes as
	$$\tor[q][\Br(E)] =  \tor[q][\Br(k)] \oplus I$$
	and every element in $I$ can be expressed as
	$$\Cor\left( \left( l^q\right)^i, n_Q \right)_{L'
		(E)} \otimes  \Cor(a,t_P)_{L'(E)}$$
	for some $0 \leq i < q$ and $a \in L'^\times$.
\end{prop}

Suppose that $q$ divides $[L:k]$ and use the notation used in \Cref{sec:Generators:qdividesLk}. Recall that every element in $I$ can be written as $\Cor(A)$ for some $A \in \tor[q][\Br(E\times \Spec(L'))]$. Such an element is trivial if and only if it is similar to $\epsilon_{L'} \circ \delta_{L'}$. 
Remark that some corestrictions of element in $\tor[q][\Br(E_L)]$ may coincide and we do not account for this in our description. 

\section{Examples}\label{ch:Examples} 

In this section, we calculate the $q$-torsion of the Brauer group for some elliptic curves $E$, where $q$ is an odd prime. For computational reasons, we only consider the case $q=3$. As before, we will study various cases depending on the extension $L$, that is the smallest Galois extension of $k$, so that $M$ is $L$-rational. 

\subsection{$M$ is $k$-rational over a number field}\label{sec:Ex:split}
Let $k = \Q(\omega)\subset \mathbb{C}$, where $\omega = -\frac12 +\frac{\sqrt{3}}{2} i$ is a primitive third root of unity. In \cite{Paladino2010} the author describes a family of elliptic curves such that $M$ is $\mathbb{Q}(\omega)$-rational, for example $E$ given by the affine equation $y^2 = x^3 + 16$. In this case, the three torsion of $E$ is generated by 
$P= (0,4)$ and $Q= (-4, 8 \omega + 4)=(-4,4\sqrt{3}i)$. Furthermore, the tangent lines at $P$ and $Q$, respectively, are given by $t_P = y-4$ and $$t_Q= y - \frac{6}{2 \omega + 1 } (x+4) - 8 \omega - 4 = y + 2 \sqrt{3}i x +4 \sqrt{3} i.$$ 
By the previous discussion $\tor[3][\Br(E)] = \tor[3][\Br(k)] \oplus I$ and every element in $I$ can be written as a tensor product 
\begin{equation}\label{eq:Ex1}
(a, y-4)_{3,k(E)} \otimes\left(b,  y + 2 \sqrt{3}i x + 4 \sqrt{3} i\right)_{3,k(E)}
\end{equation} for some $a,b \in k^\times.$
We calculate with magma, that $E(k) = M$ and therefore also $E(k)/3E(k) = M$. We now compute directly that \begin{align*} 
t_Q(P) &=  4 + 4\sqrt{3} i & \frac{ t_P(P\oplus Q)}{t_P(Q)} &= \frac{4 \sqrt{3} i - 4}{ 4\sqrt{3} i - 4} = 1\\
\frac{t_Q(P\oplus Q)}{t_Q(Q)} &= \frac{12 \sqrt{3} i + 12}{4+4\sqrt{3} i} = 3 & t_P(Q) &= 4\sqrt{3} i - 4 
\end{align*} 
Using the algorithm, we see that a tensor product as in (\ref{eq:Ex1}) is trivial if and only if $(a,b) \in k^\times/(k^\times)^3 \times k^\times/(k^\times)^3$ is in the subgroup generated by $\left( 4+ 4 \sqrt{3} i, 1\right) $ and $\left( 3, 4\sqrt{3}i-4 \right)$.

\subsection{Degree $L/k$ coprime to $q$ for $k$ a number field}\label{subsec:Ex degree L/k coprime to 3} 

Let $k = \mathbb{Q}(\omega)$ and $E$ the elliptic curve given by the affine equation 
$$y^2 = x^3 + B,$$
where $B \equiv 2 \mod (\mathbb{Q}^\times )^3$ and $B \not\equiv 1,-3 \mod (\mathbb{Q}^\times)^2$. By \cite[Theorem 3.2 and Corollary 3.3]{Bandini2012} we see that $L = k(\sqrt{B})$. Let $\sigma$ be given by $\sigma(\sqrt{B}) = - \sqrt{B}$. The three torsion of $E$ has generators $P$ and $Q$ with $P = (0,\sqrt{B})$ and $Q =  \left(\sqrt[3]{-4B},\sqrt{-3B}\right),$ so that $\sigma(P) = 2P$ and $\sigma(Q) = 2Q$. We need to calculate 
$$\Cor_{L(E)/k(E)}\left((a,t_P)_{3,L(E)} \otimes (b,t_Q)_{3,L(E)}\right).$$ 
Recall that by \Cref{rem:image of restriction} it will be enough to consider $(a,b)$ a norm in $\frac{L^\times}{(L^\times)^3} \times \frac{L^\times}{(L^\times)^3}$. For $(a,b) \in \frac{L^\times}{(L^\times)^3} \times \frac{L^\times}{(L^\times)^3}$ we have $N_{L/k}(a,b) = \left(\frac{a}{\sigma(a)}, \frac{b}{\sigma(b)}\right)$. Therefore we may assume that $N_{L/k}(a) = 1$ and $N_{L/k}(b) = 1$. We compute $$t_P = y - \sqrt{B}$$ and $$t_Q = y - \sqrt{-3B} - \frac{ 3 \left( \sqrt[3]{-4B}\right)^2 }{2 \sqrt{-3B}} \left( x - \sqrt[3]{-4B}\right) = 
 y + \sqrt{-3B} + \frac{ \sqrt[3]{2B^2} \sqrt{-3B}}{B} x.$$  

Let $a = a_1\sqrt{B} + a_2$ with $N_{L/k}(a) = 1$. Remark that if $a_1 = 0$, then $1= N_{L/k}(a) = a_2^2$ so that $\left(a,t_P\right)_{3,L(E)}$ is trivial. If $a\neq 0$, we use the algorithm given in \cite[Section 3]{Rosset1983} to calculate the corestriction explicitly as
\begin{align*} 
\Cor(a,t_P)_{3,L(E)} &= \left( a_2 + a_1y, a_1^2\right)_{3,k(E)}.
\end{align*}
Finally, let $b= b_1 \sqrt{B} + b_2$ with $N_{L/k}(b) = 1$. Similarly as before $\left(b,t_Q\right)_{3,L(E)}$ is trivial if $b_1 =0$, thus we may assume that $b_1 \neq0$. Then we compute that 
\begin{align*} 
\Cor(b, t_Q)_{3,L(E)} &= \left( b_2 - \frac{ y b_1 B}{\sqrt{-3} \left( B + \sqrt[3]{2B^2} x\right)} , 1- b_2^2 - \frac{ b_1^2 B^2 \left( x^3+B \right) }{3 \left( B + \sqrt[3]{2B^2} x\right)^2 } \right)_{3,k(E)}.
\end{align*} 
Overall, the three torsion of the Brauer group decomposes as
	$$\tor[3][\Br(E)] = \tor[3][\Br(k)]\oplus I$$
	and every element in $I$ can be written as a tensor product of at most two symbol algebras of the form 
	\begin{enumerate} 
		\item $\left( a_2 + a_1y, a_1^2\right)_{3,k(E)}$ for some $a_1, a_2 \in k$ with $a_1 \neq 0$ and $a_2^2 - B a_1^2 = 1$, and 
		\item $ \left( b_2 - \frac{ y b_1 B}{\sqrt{-3} \left( B + \sqrt[3]{2B^2} x\right)} , 1- b_2^2 - \frac{ b_1^2 B^2 \left( x^3+B \right) }{3 \left( B + \sqrt[3]{2B^2} x\right)^2 } \right)_{3,k(E)}$ for some $b_1, b_2 \in k$ with $b_1 \neq 0$ and $b_2^2 - B b_1^2 = 1$.
	\end{enumerate} 

To calculate the relations we need to specify $B$. Consider the case $B= -1024$. Using magma we calculate that $E(k) = 0$. Thus there are no additional relations. Note that some elements might still become trivial due to the fact that the corestriction is not injective. 

\subsection{Degree $L/k= q$ for $k$ a number field}\label{sec:Ex:degreeq}
Let $k = \Q(\omega)$, $\omega = -\frac{1}{2} + \frac{i\sqrt{3}}{2}$ and let $E$ be the elliptic curve given by the affine equation 
$$y^2 = x^3+4.$$
Generators of the three torsion are given by $P= (0, 2)$ and $Q=(-2\sqrt[3]{2} , 2i\sqrt{3})$. Let $l = \sqrt[3]{2}$. In our previous notation $L= k\left(\sqrt[3]{2}\right)$ and the Galois group $\Gal(L/k)$ is generated by $\overline{\sigma}$ with $\sigma(Q) = P+Q = (-\omega 2 \sqrt[3]{2}, 2i \sqrt{3})$. It can be seen that 
\begin{align*}
t_P & = y - 2 \\
t_Q & = y + i \sqrt{3} \sqrt[3]{4}x + 2 i \sqrt{3}\\
n_Q &= y - 2 \sqrt{3}i. 
\end{align*}
Therefore the $3$-torsion of $\Br(E)$ is $$\tor[3][\Br(E)] = \tor[3][\Br(k)] \oplus I$$ and every element in $I$ can be written as a tensor product 
	$$\left( 2^i, y - 2 \sqrt{3}i \right)_{3,k(E)} \otimes (a,y-2)_{3,k(E)}$$ for some $a \in k^\times$ and $0 \leq i <3$. 
We calculate that $E(k) = \left<P\right>$ and $E(L) \cong \Z/6\Z \times \Z/6\Z$. Therefore 
$E(k)/3E(k) = \left< P\right>$ and $E(L)/3E(L) = M$ and the quotient $\frac{E(k) \cap [3] E(L)}{[3]E(k)}$ is trivial. Therefore the symbol algebra 
$\left(  2, y - 2 \sqrt{3}i \right)_{k(E)}$ is not trivial. Finally, 
$$\delta\circ \res(P) = \left( 2 + 2i \sqrt{3}, 1\right) \qquad \delta \circ \res(2P) = \left( -2 + 2i\sqrt{3},1\right)$$ and therefore a symbol algebra $(a,y-2)_{k(E)}$ is trivial if and only if $$a \equiv 1, 2+2i\sqrt{3}, \text{ or }-2+2i\sqrt{3} \mod (L^{\times})^3.$$

\subsection{Over a local field}\label{sec:local} 

Denote by $\Q_7$ the $7$-adic numbers. Note that since three divides $7-1$, the field $\Q_7$ contains a primitive third root of unity $\omega$. Let $E$ be the elliptic curve 
$$E: y^2 = x^3 + 16$$ over $k$. Consider the reduction $\tilde{E}$ of $E$ modulo $7$. Then $\tilde{E}$ is a non-singular curve and using magma we see that 
$$\tilde{E}\left(\mathbb{F}_7\right) = \Z/3\Z \oplus \Z/3\Z = \left\{ 
0, (0,3), (0,4), (3,1), (3,6), (5,1),(5,6), (6,1), (6,6)
\right\}. $$
Denote by $\hat{E}$ the formal group associated to $E$ and consider the group $\hat{E}\left(7\Z_7\right)$. By \cite[IV Theorem 6.4]{silverman} there is an isomorphism 
$\hat{E} \left( 7 \Z_7 \right) \rightarrow \hat{\mathbb{G}}_a \left( 7 \Z_7 \right)$, where ${\mathbb{G}_a}$ denotes the additive group. By \cite[IV.3 and VII.2]{silverman} there is an exact sequence 
$$\xymatrix{ 0 \ar[r]&\hat{\mathbb{G}}_a\left(7\Z_7\right) \ar[r]& E(\Q_7) \ar[r]&  \tilde{E}\left(\mathbb{F}_7\right) \ar[r]& 0}.$$
Furthermore, by \cite[VII.3 Proposition 3.1]{silverman} the reduction map $\tor[3][E(\Q_7)] \rightarrow \tilde{E}\left(\mathbb{F}_7\right)$ is injective. Thus $E$ has $k$-rational $3$-torsion. Since $3$ is a unit in $\Z_7$, we further deduce that $E(\Q_7)/[3]E(\Q_7) = \tilde{E}(\mathbb{F}_7)/[3]\tilde{E}(\mathbb{F}_7) =M.$ Finally, 
$$\Q_7^\times / \left( \Q_7^\times \right)^3 \cong \left( \Z_7^\times \times \left(7\Z_7\right) \right)/ \left( \Z_7^\times \times \left(7\Z_7\right) \right)^3 \cong \Z/3\Z \times \Z/3\Z.$$
Therefore, $H^1(k,M) \cong \left( \Q_7^\times / \left( \Q_7^\times \right)^3  \right)^2 \cong \left( \Z/3\Z \right)^4$. 
By the algorithm and using \cite[Corollaire 2.3]{GroupedeBrauerIII}, the three torsion of the Brauer group decomposes as follows 
$$\tor[3][\Br(E)] = \tor[3][\Br\left(\Q_7\right)] \oplus \left( \Z/3\Z  \right)^2 = \tor[3][\left( \Q/\Z \right)] \oplus \left( \Z/3\Z \right)^2 
= \left( \Z/3\Z \right)^3.
$$ 

\begin{rem}
	The above computations also show that $\tor[3][\Br\left( \tilde{E}\right)] = \tor[3][\Br\left( \mathbb{F}_7 \right)] = 0$. 	
\end{rem}

\bibliographystyle{abbrv}
\bibliography{bibl}

\end{document}